\documentclass[12pt, reqno]{amsart}
\usepackage{amsmath, amsthm, amscd, amsfonts, amssymb, graphicx, color, mathrsfs}
\usepackage[bookmarksnumbered, colorlinks, plainpages]{hyperref}
\usepackage[all]{xy} 
\usepackage{slashed}

\newtheorem{theorem}{Theorem}[section]
\newtheorem{lemma}[theorem]{Lemma}

\theoremstyle{definition}
\newtheorem{definition}[theorem]{Definition}
\newtheorem{example}[theorem]{Example}

\theoremstyle{remark}
\newtheorem{remark}[theorem]{Remark}
\numberwithin{equation}{section}

\begin{document}
\setcounter{page}{1}

\title[$L^p$-$L^q$ boundedness of pseudo-differential operators  ]{$L^p$-$L^q$ boundedness of pseudo-differential operators on graded Lie groups}

 \author[D. Cardona]{Duv\'an Cardona}
\address{
 Duv\'an Cardona:
  \endgraf
  Department of Mathematics: Analysis, Logic and Discrete Mathematics
  \endgraf
  Ghent University, Belgium
  \endgraf
  {\it E-mail address} {\rm duvan.cardonasanchez@ugent.be}
  }

  \author[V. Kumar]{Vishvesh Kumar}
\address{
 Vishvesh Kumar:
  \endgraf
  Department of Mathematics: Analysis, Logic and Discrete Mathematics
  \endgraf
  Ghent University, Belgium
  \endgraf
  {\it E-mail address} {\rm vishveshmishra@gmail.com, Vishvesh.Kumar@UGent.be}
  }

\author[M. Ruzhansky]{Michael Ruzhansky}
\address{
  Michael Ruzhansky:
  \endgraf
  Department of Mathematics: Analysis, Logic and Discrete Mathematics
  \endgraf
  Ghent University, Belgium
  \endgraf
 and
  \endgraf
  School of Mathematical Sciences
  \endgraf
  Queen Mary University of London
  \endgraf
  United Kingdom
  \endgraf
  {\it E-mail address} {\rm michael.ruzhansky@ugent.be}
  }

 \allowdisplaybreaks



\thanks{The authors are supported  by the FWO  Odysseus  1  grant  G.0H94.18N:  Analysis  and  Partial Differential Equations and by the Methusalem programme of the Ghent University Special Research Fund (BOF)
(Grant number 01M01021). The last two authors are supported by  FWO Senior Research Grant G011522N. The third author is also supported  by EPSRC grants 
EP/R003025/2 and EP/V005529/1.}

\begin{abstract} In this paper we establish the   $L^p$-$L^q$ estimates for global pseudo-differential operators on graded Lie groups. We provide both necessary and sufficient conditions for the $L^p$-$L^q$ boundedness of  pseudo-differential operators associated with the global H\"ormander symbol classes on graded Lie groups, within the range $1<p\leq 2 \leq q<\infty$.  Additionally, we present a sufficient condition for the $L^p$-$L^q$ estimates of pseudo-differential operators within the range $1<p\leq q\leq 2$ or $2\leq p\leq q<\infty$. The proofs rely on estimates of the Riesz and Bessel potentials associated with Rockland operators, along with previously established results on $L^p$-boundedness of global pseudo-differential operators on graded Lie groups. Notably, as a byproduct, we also establish the sharpness of the Sobolev embedding theorem for the inhomogeneous Sobolev spaces on graded Lie groups.

\end{abstract} \maketitle

\tableofcontents
\section{Introduction} 
\subsection{Outline}
The boundedness of pseudo-differential operators is a fundamental problem in harmonic analysis that arises due to their diverse applications in harmonic analysis and partial differential equations \cite{HormanderBook34, Tay}. Several contributions have been made in the Euclidean space setting towards establishing the $L^p$-boundedness of pseudo-differential operators associated with the Hörmander symbol class $S^m_{\rho, \delta}(\mathbb{R}^n \times \mathbb{R}^n)$ \cite{B79, F73, Tay}. For non-commutative groups, recent papers \cite{CDR21JGA,DR19} have investigated the $L^p$-boundedness of pseudo-differential calculus associated with global H\"ormander classes on compact Lie groups and graded Lie groups. Furthermore, H\"ormander \cite{Hor67} established the $L^p$-$L^q$ boundedness of pseudo-differential operators associated with the symbol class $S^m_{\rho, \delta}(\mathbb{R}^n \times \mathbb{R}^n)$ for the range $1<p\leq q<\infty$ (see also \cite{Hounie}).   It is worth noting that H\"ormander provided a sufficient condition for the $L^p$-$L^q$ boundedness of Fourier multipliers in his seminal paper \cite{Hormander1960} for the range $1<p\leq 2 \leq q<\infty$, and this result has recently been extended to various settings such as compact homogeneous manifolds \cite{ARN1,ARN}, graded Lie groups \cite{RR22},  general locally compact unimodular groups \cite{AR, RR23}, quantum groups \cite{AMR}, and for the generalised Fourier transform on $\mathbb{R}^n$ and hypergroups \cite{KR, KRIMRN}, as well as for eigenfunction expansions of certain operators \cite{CK1, CKNR}. However, to the best of the authors' knowledge, no attempts have been made to establish the $L^p$-$L^q$ boundedness of pseudo-differential operators on noncommutative Lie groups. The primary objective of this article is to fill this gap and make progress towards the $L^p$-$L^q$ boundedness of pseudo-differential operators for the range $1<p, q<\infty$ associated with the global H\"ormander symbol classes $S^m_{\rho, \delta}(G \times \widehat{G})$ on graded Lie groups $G$ with its unitary dual $\widehat{G},$ introduced by V. Fischer and  the third author in the monograph  \cite{FischerRuzhanskyBook}.

Graded Lie groups, which include significant examples like the Heisenberg group, Heisenberg-type groups, and stratified Lie groups, are a notable class within the realm of homogeneous nilpotent Lie groups \cite{FS82, FischerRuzhanskyBook}. What sets them apart from other nilpotent Lie groups is the existence of hypoelliptic left-invariant homogeneous partial differential operators known as Rockland operators \cite{Rock78,terRob97}. The Helffer and Nourrigat solution to the Rockland conjecture sheds light on this distinguishing characteristic \cite{HN}.   These groups have significant applications in analysis, representation theory, and geometry \cite{Howe80}. A visionary program for studying analysis and PDEs on graded Lie groups was presented by Stein in his 1970 lecture at the ICM in Nice \cite{Stein71}.  In the influential paper by Rothschild and Stein \cite{RS76}, it was demonstrated how one can learn a great deal about
certain differential operators on manifolds, by approximation  using  operators on certain
homogeneous nilpotent groups \cite{F77, Roth83}. Inspired by the setup of global pseudo-differential operators on the Heisenberg group $\mathbb{H}^n$ suggested  by M. Taylor \cite{Tay1}, V. Fischer and the third author \cite{FischerRuzhanskyBook} developed a Kohn-Nirenberg type calculus on general graded Lie groups for a suitably defined operator-valued global H\"ormander symbols classes $S^m_{\rho, \delta}(G \times \widehat{G})$ on  a graded Lie group $G$. Recently, this calculus was extended by Federico, Rottensteiner, and the third author \cite{FRR23} using the so-called $\tau$-quantisation on graded Lie groups, which, in particular, gives rise to Weyl-type pseudo-differential calculus on graded Lie groups. This construction is based on the general setup of pseudo-differential operators on unimodular type-I locally compact groups \cite{MR17}.    The $L^p$-boundedness of pseudo-differential operators associated with operator-valued H\"ormander symbol class $S^m_{\rho, \delta}(G \times \widehat{G})$ was established  in \cite{CDR21JGA} and its validity for $\tau$-quantisations was discussed in \cite{FRR23}. We also refer to \cite{CDRFC} for further developments and applications to PDEs.

\subsection{Main results}
To state our main results, some notation, and basic knowledge should be in order. We refer to Section \ref{sec2} for a detailed exposition. Let $G$ be a graded Lie group of homogeneous dimension $Q.$
We denote the positive Rockland operator on $G$ by $\mathcal{R},$ that is, a positive left-invariant hypoelliptic differential operator which is homogeneous. The existence of such an operator on $G$ is assured by the resolution of the Rockland conjecture by Helffer and Nourrigat \cite{HN}. Some examples of the Rockland operators in particular situations are as follows:
\begin{itemize}
    \item For the Heisenberg group $\mathbb{H}^n:$ the positive sub-Laplacian $(-\Delta_{\mathbb{H}^n})$ and its powers.
    \item For a stratified Lie group $G:$ let $X_1,X_2,\ldots, X_m$ be a basis of the first stratum $\mathfrak{g}_1$ of the stratified Lie algebra $\mathfrak{g}$ of $G,$ then the operator 
    $$\mathcal{R}:=(-1)^k \sum_{i=1}^m c_i X_i^{2k},\,\,c_i>0$$ is an example of a Rockland operator for any $k \in \mathbb{N}.$ Note that the case $k=1$ corresponds to  H\"ormander type sub-Laplacians.
    \item In a more general setup, that is, on a graded Lie group $G$ with a basis $X_1,X_2,\ldots,X_n$  of the Lie algebra $\mathfrak{g}$ of $G$ satisfying $D_r X_i=r^{\nu_i}X_i,\,i=1,\ldots, n,\,r>0,$ with the dilation weights $\nu_1,\nu_2,\ldots,\nu_n,$ denoting by $\nu_0$ any common multiple of $\nu_1,\nu_2,\ldots,\nu_n,$ the operator 
$$\mathcal{R}=\sum_{i=1}^n(-1)^{\frac{\nu_0}{\nu_i}}c_i X^{2 \frac{\nu_0}{\nu_i}},\,c_i>0$$ is a Rockland operator of homogeneous degree $2\nu_0$ on the graded Lie group $G.$
\end{itemize}
There have been several studies dealing with function spaces such as homogeneous and inhomogeneous Sobolev spaces on graded Lie groups associated with Rockland operators. We refer to \cite{FRsobo,FischerRuzhanskyBook,Nikolskii} and references therein for more detail. In particular, a sufficient condition for the $L^p$-$L^q$ boundedness of the Riesz operator $I_a:=\mathcal{R}^{-\frac{a}{\nu}}$ with $0<a<Q$ is given in \cite{FischerRuzhanskyBook}.

The following theorem presents the result that establishes a sufficient condition and in some cases a necessary condition, for the $L^p$-$L^q$ boundedness of pseudo-differential operators on graded Lie groups associated with the global H\"ormander symbol classes $S^m_{\rho, \delta}(G \times \widehat{G}).$ We denote the class of operators associated with $\textnormal{Op}({S}^{m}_{\rho,\delta}(G\times \widehat{G}))$ by $\Psi^{m}_{\rho,\delta}(G\times \widehat{G}).$ 
\begin{theorem}\label{main:theorem} 
Let $1<p, q<\infty.$ Let $G$ be a graded Lie group of  homogeneous dimension $Q$ and let $0\leq \delta\leq \rho\leq 1,$ $\delta\neq 1.$ Then, the following statements hold.
\begin{itemize}
    \item Let $1<p\leq2\leq  q<\infty.$ Every pseudo-differential operator $A\in S^{m}_{\rho,\delta}(G\times \widehat{G})$ admits a bounded extension from $L^p(G)$ into $L^q(G),$ that is
\begin{equation}\label{Lp-Lq:bound}
   \forall f\in C^{\infty}_0(G),\, \Vert A f\Vert_{L^q(G)}\leq  C\Vert  f\Vert_{L^p(G)}\,\,\,
\end{equation} holds, if and only if, 
\begin{equation}\label{Necessary:condition:3:intro}
   m\leq -Q\left(\frac{1}{p}- \frac{1}{q}\right).
\end{equation}  
\item Every pseudo-differential operator $A\in S^{m}_{\rho,\delta}(G\times \widehat{G})$ admits a bounded extension from $L^p(G)$ into $L^q(G),$ that is \eqref{Lp-Lq:bound} holds, in the following cases: 
\begin{itemize}
        \item[(i)] if $1<p\leq q \leq 2$ and  \begin{equation}\label{Necessary:condition:4:intro}
   m\leq -Q \left( \frac{1}{p}-\frac{1}{q}+(1-\rho) \left(\frac{1}{q}-\frac{1}{2}\right)\right). \end{equation}    
\item[(ii)] if $2 \leq p \leq q<\infty$ and 

\begin{equation}\label{Necessary:condition:3:intro:}
   m\leq -Q\left( \frac{1}{p}-\frac{1}{q}+(1-\rho) \left(\frac{1}{2}-\frac{1}{p}\right)\right).
   \end{equation}
\end{itemize}
\end{itemize}
\end{theorem}
\begin{remark} \label{rem1}  The order conditions in \eqref{Necessary:condition:3:intro},  \eqref{Necessary:condition:4:intro} and \eqref{Necessary:condition:3:intro:} can be written in a simplified way for $1<p,q<\infty$ as follows:
\begin{equation}\label{order:condition}
    m\leq -Q\left(   \frac{1}{p}-\frac{1}{q}+(1-\rho)\max\left\{ \frac{1}{2}-\frac{1}{p},\frac{1}{q}-\frac{1}{2},0\right\}\right),
\end{equation}where $Q$ is the homogeneous dimension of the group $G$.    
\end{remark}
\begin{remark}
    If $G=\mathbb{R}^n,$ the order condition in \eqref{order:condition} is sharp for Fourier multipliers, see H\"ormander \cite[Page 163]{Hor67}.
\end{remark}
\begin{remark}
    For the proof of the necessary and sufficient condition in \eqref{Necessary:condition:3:intro}, an essential tool is the sharpness of the Sobolev embedding theorem on graded Lie groups. More precisely, in Lemma \ref{Lemma:Bessel:potential:lplq} we prove that for $1<p,q<\infty,$ and for a Rokland operator  $\mathcal{R}$ of homogeneous degree $\nu>0$ on a graded Lie group $G,$ the Bessel operator $$B_{a}=(1+\mathcal{R})^{-\frac{a}{\nu}},$$ admits a bounded extension from $L^p(G)$ into $L^q(G),$ that is, the estimate
\begin{equation}\label{eqref:lplq:intro}
    \Vert B_{a} f\Vert_{L^q}\leq  C\Vert  f\Vert_{L^p}\,\,\,
\end{equation} holds, if and only if, $1<p<q<\infty$ and
\begin{equation}\label{Necessary:condition:2:intro}
   a\geq Q\left(\frac{1}{p}- \frac{1}{q}\right).
\end{equation}That \eqref{Necessary:condition:2:intro} is a sufficient condition for the boundedness of $B_a$ was proved in \cite{FischerRuzhanskyBook}. Our  contribution in Lemma \ref{Lemma:Bessel:potential:lplq} is then the proof of the {\it reverse} statement.   
\end{remark}

\section{Preliminaries} \label{sec2}
We begin by reviewing some basic definitions and results from
the theory of pseudo-differential operators  on graded Lie groups as developed  in \cite{FischerRuzhanskyBook} used in our analysis of $L^p$-$L^q$-bounds in the setting of graded Lie groups.  We introduce the necessary tools on the subject in the next subsections.
\subsection{H\"ormander classes on graded Lie groups} For the facts below we follow \cite{FischerRuzhanskyBook}. 

Let $G$ be a connected, simply connected nilpotent Lie group with Lie algebra
$\mathfrak{g}.$ We  always  assume that  $\mathfrak{g}$ is a real vector space of finite dimension. In this case, the exponential map 
$$\textnormal{exp}: \mathfrak{g}\rightarrow G  $$
is a global diffeomorphism that allows us to make the identification $G\simeq \mathbb{R}^n.$
The corresponding mapping $(x,y)\mapsto x\cdot y$ is a polynomial mapping and the Haar measure on $G$ is induced by the exponential mapping from the Lebesgue measure on $\mathbb{R}^n,$ where $n=\dim(G)$ is the topological dimension of $G.$

A {\it family of dilations } on the Lie algebra $\mathfrak{g}$ is a family of endomorphisms $D_{r}\in \textnormal{End}(\mathfrak{g}),$ $r>0,$ satisfying the  following properties:
\begin{itemize}
    \item $\forall r>0,$ $D_r$ is an  algebra automorphism, i.e. if $[\cdot,\cdot]$ denotes the Lie bracket on $\mathfrak{g},$ then, for any $r>0,$ we have that $$D_{r}(X+Y)=D_{r}X+D_{r}(Y),$$ and $$D_{r}[X,Y]=[D_rX,D_rY],$$ for all $X,Y\in \mathfrak{g}.$
    \item There exists a diagonalisable linear operator $A:\mathfrak{g}\rightarrow \mathfrak{g} $ such that $$D_{r}X=e^{\ln(r)A}X,\,\,\forall X\in \mathfrak{g}.$$ Such $A$ is called the {\it  dilation matrix} of the group.
\end{itemize}
A {\it homogeneous group} is a connected, simply connected nilpotent Lie group $G$
such that the Lie algebra $\mathfrak{g}$ is endowed with a family of dilations $\{D_{r}\}_{ r>0}.$ In this case the mappings
$$\textnormal{exp}\circ D_{r}\circ \textnormal{exp}^{-1}: G\rightarrow G  $$ are group automorphisms on $G$ that we also denote by $D_r.$ Moreover, if $$\textnormal{Spectrum}(A)=\{\nu_1,\cdots \nu_n\}$$ is the set of eigenvalues of the dilation matrix $A$, the number 
$$Q=\textnormal{Tr}(A):=\nu_1+\cdots+\nu_n$$ is called the {\it homogeneous dimension of } $G.$

Let  $\pi$ be  a continuous, unitary and irreducible  representation of $G$ on a separable Hilbert space $H_{\pi},$ this means that,
\begin{itemize}
    \item[1.] $\pi\in \textnormal{Hom}(G, \textnormal{U}(H_{\pi})),$  i.e. $\pi(gg')=\pi(g)\pi(g')$ and for the  adjoint of $\pi(g),$ $\pi(g)^*=\pi(g^{-1}),$ for every $g,g'\in G.$
    \item[2.] The function $(g,v)\mapsto \pi(g)v, $ from $G\times H_\pi$ into $H_\pi$ is a continuous mapping.
    \item[3.] For every $g\in G,$ and for a vector subspace $W_\pi\subset H_\pi,$ if $\pi(g)W_{\pi}\subset W_{\pi},$ then $W_\pi=H_\pi$ or $W_\pi=\{0\}.$
\end{itemize} Let $\textnormal{Rep}(G)$ be the family of unitary, continuous and irreducible representations of $G.$ The relation, 
\begin{equation*}
    \pi_1\sim \pi_2\Longleftrightarrow \exists B\in \mathscr{B}(H_{\pi_1},H_{\pi_2}): \,\forall g\in G, \, B\pi_{1}(g)B^{-1}=\pi_2(g) 
\end{equation*}  is an equivalence relation and the unitary dual of $G,$ denoted by $\widehat{G}$ is defined via
$
    \widehat{G}:={\textnormal{Rep}(G)}/{\sim}.
$ We always identify any representation with its equivalence class $\pi\simeq [\pi].$ One reason to do this is that in the setting of nilpotent Lie groups the unitary dual is a continuous set.

We are going to denote by  $d\pi$ the {\it Plancherel measure} on $\widehat{G}.$ 
The group Fourier transform $\mathscr{F}_{G}(f):=\widehat{f}$ of $f\in \mathscr{S}(G)\cong \mathscr{S}(\mathbb{R}^n), $  at $\pi\in\widehat{G},$ is defined by 
\begin{equation*}
    \widehat{f}(\pi)=\int\limits_{G}f(g)\pi(g)^*dg:H_\pi\rightarrow H_\pi,\textnormal{   and   }\mathscr{F}_{G}:\mathscr{S}(G)\rightarrow \mathscr{S}(\widehat{G}):=\mathscr{F}_{G}(\mathscr{S}(G)).
\end{equation*}
Under the identification, $\pi\simeq  [\pi]=\{\pi':\pi\sim \pi'\}$,  the Kirillov trace character $\Theta_\pi$ defined by  $$f\in \mathscr{S}(G)\mapsto (\Theta_{\pi},f):
=\textnormal{ Tr }(\widehat{f}(\pi)),$$ is a tempered distribution on $G.$ In particular, the identity
$$
    f(e_G)=\int\limits_{\widehat{G}}\textnormal{ Tr }(\widehat{f}(\pi))d\pi,
$$ 
implies the {\it Fourier inversion formula} $f=\mathscr{F}_G^{-1}(\widehat{f}),$ where
\begin{equation*}
    (\mathscr{F}_G^{-1}\sigma)(g):=\int\limits_{\widehat{G}}\textnormal{ Tr }[\pi(g)\sigma(\pi)]d\pi,\,\,g\in G,\,\,\,\,\mathscr{F}_G^{-1}:\mathscr{S}(\widehat{G})\rightarrow\mathscr{S}(G),
\end{equation*}is the {\it inverse Fourier  transform}. We recall that  the Plancherel theorem takes the form $$\Vert f\Vert_{L^2(G)}=\Vert \widehat{f}\Vert_{L^2(\widehat{G})},$$  where we have denoted by   $$ L^2(\widehat{G}):=\int\limits_{\widehat{G}}H_\pi\otimes H_{\pi}^*d\pi,$$ the Hilbert space endowed with the norm: $$\Vert \sigma\Vert_{L^2(\widehat{G})}=\left(\int_{\widehat{G}}\Vert \sigma(\pi)\Vert_{\textnormal{HS}}^2d\pi\right)^{\frac{1}{2}}.$$

An important class of operators that interact with the dilations on the group $G$ are  {\it homogeneous operators}.  We recall that a linear operator $T:C^\infty(G)\rightarrow \mathscr{D}'(G)$ is {\it homogeneous} of  degree $\nu\in \mathbb{C}$ if for every $r>0$ the identity
\begin{equation*}
T(f\circ D_{r})=r^{\nu}(Tf)\circ D_{r}
\end{equation*}
holds for every $f\in \mathscr{D}(G). $

For every unitary representation $\pi\in\widehat{G},$ acting as $\pi:G\rightarrow U({H}_{\pi}),$ on a representation space ${H}_{\pi},$ we denote by ${H}_{\pi}^{\infty}$ the set of {\it smooth vectors} on ${H}_{\pi},$ that consists of all elements $v\in {H}_{\pi}$ such that the function $x\mapsto \pi(x)v,$ $\pi\in \widehat{G},$ is smooth.

A {\it Rockland operator} is a positive left-invariant differential operator $\mathcal{R}$ which is homogeneous of positive degree $\nu=\nu_{\mathcal{R}}$ and such that, for every unitary irreducible non-trivial representation $\pi\in \widehat{G},$ $\pi(\mathcal{R})$ is an injective operator on the space of smooth vectors ${H}_{\pi}^{\infty};$ Here, $\sigma_{\mathcal{R}}(\pi)=\pi(\mathcal{R})$ is the symbol associated to $\mathcal{R}.$ It coincides with the infinitesimal representation of $\mathcal{R}$ as an element of the universal enveloping algebra.

It can be shown that a Lie group $G$ is graded if and only if there exists a differential Rockland operator on $G.$ 
\begin{remark}
If the Rockland operator $$\mathcal{R}:=\sum_{[\alpha]=\nu}a_{\alpha}X^{\alpha}$$ is formally self-adjoint, then $\mathcal{R}$ and $$\pi(\mathcal{R})=\sum_{[\alpha]=\nu}a_{\alpha}\pi(X^{\alpha})$$ admit self-adjoint extensions on $L^{2}(G)$ and on ${H}_{\pi},$ respectively, for a.e. $\pi\in \widehat{G}$. Now if we use the same notation for their (unbounded) self-adjoint
extensions and we denote by $E$ and $E_{\pi}$  their respective spectral measures, we will denote by
$$ \phi(\mathcal{R})=\int\limits_{-\infty}^{\infty}\phi(\lambda) dE(\lambda),\,\,\,\textnormal{and}\,\,\,\pi(\phi(\mathcal{R}))\equiv \phi(\pi(\mathcal{R}))=\int\limits_{-\infty}^{\infty}\phi(\lambda) dE_{\pi}(\lambda), $$ the measurable functions defined by the spectral functional calculus associated with $\mathcal{R}$ and $\pi(\mathcal{R}),$ respectively. 
In general, we will reserve the notation ${E_A(\lambda)}_{0\leq\lambda<\infty}$ for the spectral measure associated with a {\it positive} and self-adjoint operator $A$ on a Hilbert space $H.$ 
\end{remark}

For  any multi-index $\alpha\in \mathbb{N}_0^n,$ and for an arbitrary family $\{X_1,\cdots, X_n\},$ of left-invariant  vector-fields  compatible with respect to the graded structure of the Lie algebra we will use the notation
\begin{equation}
    [\alpha]:=\sum_{j=1}^n\nu_j\alpha_j,
\end{equation}for the homogeneous degree of the operator $X^{\alpha}:=X_1^{\alpha_1}\cdots X_{n}^{\alpha_n}.$ 

In order to present a consistent definition of pseudo-differential operators, namely, a quantisation formula, as developed in \cite{FischerRuzhanskyBook} (see  \eqref{Quantization}), we need   to define a suitable class of Sobolev spaces on the unitary dual $\widehat{G}$ acting on the set of smooth vectors $H_{\pi}^{\infty},$ for every representation space $H_{\pi}.$ To record this notion, let $\mathcal{R}$ be  a positive Rockland operator on $G$ of homogeneous degree $\nu>0.$
 \begin{definition}[Sobolev spaces on  $H_{\pi}^{\infty}$] Let $\pi_1\in \textnormal{Rep}(G),$ and let $a\in \mathbb{R}.$ We denote by $H_{\pi_1}^a,$ the Hilbert space obtained  by the completion of $H_{\pi_1}^\infty$ with respect to the norm
 \begin{equation*}
     \Vert v \Vert_{H_{\pi_1}^a}=\Vert\pi_1(1+\mathcal{R})^{\frac{a}{\nu}} v\Vert_{H_{\pi_1}}.
 \end{equation*}   
 \end{definition}
 
 In order to introduce the general notion of a symbol, that in our contexts are operators acting on Hilbert spaces (the corresponding representation spaces), as the one developed in \cite{FischerRuzhanskyBook}, we require the following definition.
 
\begin{definition}
A $\widehat{G}$-field of operators $$\sigma=\{\sigma(\pi):\pi\in \widehat{G}\}$$ defined on the class of smooth vectors {is defined} on some Sobolev space ${H}_\pi^a$ when for each unitary representation $\pi_1\in \textnormal{Rep}(G),$ such that $\pi_1\in \pi\in \widehat{G},$ the operator $\sigma(\pi_1)$ is bounded from $H^a_{\pi_1}$ to $H_{\pi_{1}}$ in the sense that
\begin{equation*}
    \sup_{ \Vert v\Vert_{H_{\pi_1}^a}=1}\Vert \sigma(\pi_1)v \Vert_{H_\pi}<\infty,
\end{equation*}holds,  and satisfying for every two elements $\sigma(\pi_1)$ and $\sigma(\pi_2)$ in $\sigma(\pi)$ the property
\begin{equation*}
  \textnormal{If  }  \pi_{1}\sim \pi_2  \textnormal{  then  }   \sigma(\pi_1)\sim \sigma(\pi_2). 
\end{equation*}
\end{definition}
 We recall that the {\it inhomogenous Sobolev space } $L^{p}_{a}(G)$ is  defined by the completion of $\mathscr{D}(G)$ with respect to the norm (see \cite[Chapter 4]{FischerRuzhanskyBook})
\begin{equation}\label{L2ab2}
    \Vert f \Vert_{L^{p}_{a}(G)}=\Vert (1+\mathcal{R})^{\frac{a}{\nu}}f\Vert_{L^p(G)},
\end{equation} for $a\in \mathbb{R},$ and for any $0<p<\infty.$ On the other hand,  the {\it homogenous Sobolev space } $\dot{L}^{p}_{a}(G)$ is  defined by the completion of $\mathscr{D}(G)\cap \textnormal{Dom}(\mathcal{R}^{\frac{a}{\nu}}),$ with respect to the norm (see \cite[Chapter 4]{FischerRuzhanskyBook})
\begin{equation}\label{L2ab2}
    \Vert f \Vert_{\dot{L}^{p}_{a}(G)}=\Vert \mathcal{R}^{\frac{a}{\nu}}f\Vert_{L^p(G)},
\end{equation} for $a\in \mathbb{R},$ and for any $0<p<\infty.$

\begin{definition}
A $\widehat{G}$-field of operators  defined on smooth vectors {with range} in the Sobolev space $H_{\pi}^a$ is a family of classes of operators $\sigma=\{\sigma(\pi):\pi\in \widehat{G}\}$ where
\begin{equation*}
    \sigma(\pi):=\{\sigma(\pi_1):H^{\infty}_{\pi_1}\rightarrow H_{\pi}^a,\,\,\pi_1\in \pi\},
\end{equation*} for every $\pi\in \widehat{G}$ viewed as a subset of $\textnormal{Rep}(G),$ satisfying for every two elements $\sigma(\pi_1)$ and $\sigma(\pi_2)$ in $\sigma(\pi):$
\begin{equation*}
  \textnormal{If  }  \pi_{1}\sim \pi_2  \textnormal{  then  }   \sigma(\pi_1)\sim \sigma(\pi_2). 
\end{equation*}
\end{definition}
  The following notion  will be useful in order to use the general theory of non-commutative integration. 
\begin{definition}
A $\widehat{G}$-field of operators  defined on smooth vectors with range in the Sobolev space $H_\pi^a$ {is measurable}  when for some (and hence for any) $\pi_1\in \pi$ and any vector $v_{\pi_1}\in H_{\pi_1}^\infty,$ as $\pi\in \widehat{G},$ the resulting field $\{\sigma(\pi)v_\pi:\pi\in\widehat{G}\},$ 
is $d\pi$-measurable and
\begin{equation*}
    \int\limits_{\widehat{G} }\Vert v_\pi \Vert^2_{H_\pi^a}d\pi=\int\limits_{\widehat{G} }\Vert\pi(1+\mathcal{R})^{\frac{a}{\nu}} v_\pi \Vert^2_{H_\pi}d\pi<\infty.
\end{equation*}
\end{definition}
\begin{remark}
We always assume that a $\widehat{G}$-field of operators  defined on smooth vectors {with range} in the Sobolev space $H_{\pi}^a$ is $d\pi$-measurable.
\end{remark} The $\widehat{G}$-fields of operators associated to Rockland operators can be defined as follows.
\begin{definition}
Let $L^2_a(\widehat{G})$ denote the space of measurable fields of operators $\sigma$ with range in $H_\pi^a,$ that is,
\begin{equation*}
    \sigma=\{\sigma(\pi):H_\pi^\infty\rightarrow H_\pi^a\}, \textnormal{ with }\{\pi(1+\mathcal{R})^{\frac{a}{\nu}}\sigma(\pi):\pi\in \widehat{G}\}\in L^2(\widehat{G}),
\end{equation*}for one (and hence for any) Rockland operator of homogeneous degree $\nu.$ We also denote
\begin{equation*}
    \Vert \sigma\Vert_{L^2_a(\widehat{G})}:=\Vert \pi(1+\mathcal{R})^{\frac{a}{\nu}}\sigma(\pi)\Vert_{L^2(\widehat{G})}.
\end{equation*}
\end{definition} By using the notation above, we will introduce a family of function spaces required to define $\widehat{G}$-fields of operators (that will be used to define the symbol of a pseudo-differential operator, see Definition \ref{SRCK}).
\begin{definition}[The spaces $\mathscr{L}_{L}(L^2_a(G),L^2_b(G)),$ $\mathcal{K}_{a,b}(G)$ and $L^\infty_{a,b}(\widehat{G})$]
\hspace{0.1cm}
\begin{itemize}
    \item The space $\mathscr{L}_{L}(L^2_a(G),L^2_b(G)) $ consists of all  left-invariant linear operators $T$  such that  $T:L^2_a(G)\rightarrow L^2_b(G) $ extends to a bounded operator.
    \item The space  $\mathcal{K}_{a,b}(G)$ is the family of all right convolution kernels of elements in $  \mathscr{L}_{L}(L^2_a(G),L^2_b(G))  ,$ i.e. $k=T\delta\in \mathcal{K}_{a,b}(G)$ if and only if  $T\in
    \mathscr{L}_{L}(L^2_a(G),L^2_b(G)) .$ 
    \item We also define the space $L^\infty_{a,b}(\widehat{G})$ of symbols by the following condition:  $\sigma\in L^\infty_{a,b}(\widehat{G})$ if 
\begin{equation*}
    \Vert \pi(1+\mathcal{R})^{\frac{b}{\nu}}\sigma(\pi)\pi(1+\mathcal{R})^{-\frac{a}{\nu}} \Vert_{L^\infty(\widehat{G})}:=\sup_{\pi\in\widehat{G}}\Vert \pi(1+\mathcal{R})^{\frac{b}{\nu}}\sigma(\pi)\pi(1+\mathcal{R})^{-\frac{a}{\nu}} \Vert_{\mathscr{B}(H_\pi)}<\infty.
\end{equation*}
\end{itemize} In this case $T_\sigma:L^2_a(G)\rightarrow L^2_b(G)$ extends to a bounded operator with 
\begin{equation*}
  \Vert \sigma\Vert _{L^\infty_{a,b}(\widehat{G})}= \Vert T_{\sigma} \Vert_{\mathscr{L}(L^2_a(G),L^2_b(G))},
\end{equation*} and note that   $\sigma\in L^\infty_{a,b}(\widehat{G})$ if and only if $k:=\mathscr{F}_{G}^{-1}\sigma \in \mathcal{K}_{a,b}(G).$
\end{definition}
   With the previous definitions, we will introduce the type of symbols used when quantising on nilpotent Lie groups.
   
   \begin{definition}[Symbols and right-convolution kernels]\label{SRCK} A {symbol} is a field of operators $\{\sigma(x,\pi):H_\pi^\infty\rightarrow H_\pi^\infty,\,\,\pi\in\widehat{G}\},$ depending on $x\in G,$ such that 
   \begin{equation*}
       \sigma(x,\cdot)=\{\sigma(x,\pi):H_\pi^\infty\rightarrow H_\pi^\infty,\,\,\pi\in\widehat{G}\}\in L^\infty_{a,b}(\widehat{G})
   \end{equation*}for some $a,b\in \mathbb{R}.$ The {right-convolution kernel} $k\in C^\infty(G,\mathscr{S}'(G))$ associated with $\sigma$ is defined, via the inverse Fourier transform on the group by
   \begin{equation*}
       x\mapsto k(x)\equiv k_{x}:=\mathscr{F}_{G}^{-1}(\sigma(x,\cdot)): G\rightarrow\mathscr{S}'(G).
   \end{equation*}
   \end{definition}
   Definition \ref{SRCK} in this section allows us to establish the following theorem, which gives sense to the quantization of pseudo-differential operators in the setting of graded Lie groups  (see Theorem 5.1.39 of \cite{FischerRuzhanskyBook}).
   \begin{theorem}\label{thetheoremofsymbol}
   Let us consider a symbol $\sigma$ and its associated right-convolution  kernel $k.$ For every $f\in \mathscr{S}(G),$ let us define the operator $A$ acting on $\mathscr{S}(G),$ via 
   \begin{equation}\label{pseudo}
       Af(x)=(f\ast k_{x})(x),\,\,\,\,\,x\in G.
   \end{equation}Then $Af\in C^\infty,$ and 
   \begin{equation}\label{Quantization}
       Af(x)=\int\limits_{\widehat{G}}\textnormal{ Tr }(\pi(x)\sigma(x,\pi)\widehat{f}(\pi))d\pi.
   \end{equation}
   \end{theorem}
   
 Theorem   \ref{thetheoremofsymbol} motivates the following definition.
 \begin{definition}\label{DefiPSDO}
A continuous linear operator $A:C^\infty(G)\rightarrow\mathscr{D}'(G)$ with Schwartz kernel $K_A\in C^{\infty}(G)\widehat{\otimes}_{\pi} \mathscr{D}'(G),$ is a pseudo-differential operator, if
 there exists a  \textit{symbol}, which is a field of operators $\{\sigma(x,\pi):H_\pi^\infty\rightarrow H_\pi^\infty,\,\,\pi\in\widehat{G}\},$ depending on $x\in G,$ such that 
   \begin{equation*}
       \sigma(x,\cdot)=\{\sigma(x,\pi):H_\pi^\infty\rightarrow H_\pi^\infty,\,\,\pi\in\widehat{G}\}\in L^\infty_{a,b}(\widehat{G})
   \end{equation*}for some $a,b\in \mathbb{R},$ such that, the Schwartz kernel of $A$ is given by 
\begin{equation*}
    K_{A}(x,y)=\int\limits_{\widehat{G}}\textnormal{ Tr }(\pi(y^{-1}x)\sigma(x,\pi))d\pi=k_{x}(y^{-1}x).
\end{equation*}
{{In this case, we use  the notation
\begin{equation*}
    A:=\textnormal{Op}(\sigma),
\end{equation*}     
to indicate that $A$ is the pseudo-differential operator associated with symbol  $\sigma.$}}
\end{definition}
On the other hand, it is also known that one can write a global symbol in terms of its corresponding pseudo-differential operator (see \cite[Theorem 3.2]{CDR21JGA}). Indeed, let $A:C^\infty(G) \rightarrow C^\infty(G)$ be a continuous linear operator with symbol 
$$\sigma:=\{\sigma(x,\pi):H_\pi^\infty\rightarrow H_\pi^\infty,\,\,\pi\in\widehat{G},\,\,\, x\in G\}
 $$ such that 
 $$ Af(x)=\int\limits_{\widehat{G}}\textnormal{ Tr }(\pi(x)\sigma(x,\pi)\widehat{f}(\pi))d\pi$$
 for every $f \in \mathscr{S}(G)$ and a.e. $(x, \pi) \in G \times \widehat{G}.$ Then, we have 
 $$\sigma(x, \pi)=\pi(x)^* A\pi(x)$$
for every $x \in G$ and a. e. $\pi \in \widehat{G}$ provided that the operator $A\pi(x)$ is a densely defined operator on $H^\infty_\pi.$ 

The main tool in the construction of the global H\"ormander classes is the notion of difference operators. 
 Indeed, for every smooth function $q\in C^\infty(G)$ and $\sigma\in L^\infty_{a,b}(G),$ where $a,b\in\mathbb{R},$ the difference operator $\Delta_q$ acts on $\sigma$ according to the formula (see Definition 5.2.1 of \cite{FischerRuzhanskyBook}),
 \begin{equation*}
     \Delta_q\sigma(\pi)\equiv [\Delta_q\sigma](\pi):=\mathscr{F}_{G}(qf)(\pi),\,\pi\in\widehat{G},\textnormal{ where }f:=\mathscr{F}_G^{-1}\sigma\,\,.
 \end{equation*}
We will reserve the notation $\Delta^{\alpha}$ for the difference operators defined by the functions $$q_{\alpha}(x):=x^\alpha,\,\,\alpha\in \mathbb{N}_0^n,$$ while we denote  by $\tilde{\Delta}^{\alpha}$ the difference operators associated with the functions   $\tilde{q}_{\alpha}(x)=(x^{-1})^\alpha$. In particular, we have the Leibnitz rule,
\begin{equation}\label{differenceespcia;l}
    \Delta^{\alpha}(\sigma\tau)=\sum_{\alpha_1+\alpha_2=\alpha}c_{\alpha_1,\alpha_2}\Delta^{\alpha_1}(\sigma)\Delta^{\alpha_2}(\tau),\,\,\,\sigma,\tau\in L^\infty_{a,b}(\widehat{G}).
\end{equation}

In terms of these difference operators, the global H\"ormander classes introduced in \cite{FischerRuzhanskyBook} can be defined as follows.
 Let:
 \begin{itemize}
     \item $0\leq \delta,\rho\leq 1,$ and
     \item let $\mathcal{R}$ be a positive Rockland operator of homogeneous degree $\nu>0,$
 \end{itemize}
  if $m\in \mathbb{R},$ we say that the symbol $\sigma\in L^\infty_{a,b}(\widehat{G}), $ where $a,b\in\mathbb{R},$ belongs to the $(\rho,\delta)$-H\"ormander class of order $m,$ $S^m_{\rho,\delta}(G\times \widehat{G}),$ if for all $\gamma\in \mathbb{R},$ the following conditions
\begin{equation}\label{seminorm}
   p_{\alpha,\beta,\gamma,m}(\sigma)= \operatornamewithlimits{ess\, sup}_{(x,\pi)\in G\times \widehat{G}}\Vert \pi(1+\mathcal{R})^{\frac{\rho [\alpha] -\delta [\beta] -m-\gamma}{\nu}}[X_{x}^\beta \Delta^{\alpha}\sigma(x,\pi)] \pi(1+\mathcal{R})^{\frac{\gamma}{\nu}}\Vert_{\textnormal{op}}<\infty,
\end{equation}
hold true for all $\alpha$ and $\beta$ in $\mathbb{N}_0^n.$ The resulting class $S^m_{\rho,\delta}(G\times \widehat{G}),$ does not depend on the choice of the Rockland operator $\mathcal{R}.$ In particular (see Theorem 5.5.20 of \cite{FischerRuzhanskyBook}), the following facts are equivalent: 
\begin{itemize}
    \item[(A)] $\forall \alpha,\beta\in \mathbb{N}_{0}^n, \forall\gamma\in \mathbb{R}, $   $p_{\alpha,\beta,\gamma,m}(\sigma)<\infty;$
    
    \item[(B)]  $\forall \alpha,\beta\in \mathbb{N}_{0}^n, $   $p_{\alpha,\beta,0,m}(\sigma)<\infty;$
    
    \item[(C)]  $\forall \alpha,\beta\in \mathbb{N}_{0}^n, $   $p_{\alpha,\beta,m+\delta [\beta] -\rho [\alpha] ,m}(\sigma)<\infty;$
    \item[(D)] $\sigma\in {S}^{m}_{\rho,\delta}(G\times \widehat{G});$
\end{itemize}and in view of the additional Theorem 13.16 of \cite{CR20}, the conditions (A), (B), (C) and (D) are equivalent to the following one:

\begin{itemize}
    \item[(E)]  $\forall \alpha,\beta\in \mathbb{N}_{0}^n, \exists \gamma_0\in \mathbb{R}, $   $p_{\alpha,\beta,\gamma_0,m}(\sigma)<\infty.$
\end{itemize}

 We will denote,
\begin{equation}
     \Vert \sigma\Vert_{k\,,S^{m}_{\rho,\delta}}:= \max_{ [\alpha] + [\beta] \leq k}\{  p_{\alpha,\beta,0,m}(\sigma)\}.
\end{equation}
\begin{remark}
In the abelian case $G=\mathbb{R}^n,$ endowed with its natural structure of the abelian  group, and with $$\mathcal{R}=-\Delta_{x}, \,x\in \mathbb{R}^n,$$ with $\Delta_{x}=\sum_{j=1}^{n}\partial_{x_i}^{2}$ being the usual Laplace operator on $\mathbb{R}^n,$ the classes defined via \eqref{seminorm}, agree with the well known H\"ormander classes on $\mathbb{R}^n$ (see e.g. H\"ormander \cite[Vol. 3]{HormanderBook34}). In this case the difference operators are the partial derivatives on $\mathbb{R}^n,$ (see Remark 5.2.13 and Example 5.2.6 of \cite{FischerRuzhanskyBook}).  
\end{remark}
For an arbitrary graded Lie group, the H\"ormander classes $S^m_{\rho,\delta}(G\times \widehat{G}),$ $m\in\mathbb{R},$ provide a symbolic calculus closed under compositions, adjoints, and existence of parametrices. The following theorem summarises the composition and the adjoint rules for global operators  as well as the Calder\'on-Vaillancourth theorem. 
\begin{theorem}\label{calculus} Let $0\leqslant \delta<\rho\leqslant 1,$ and let us denote $\Psi^{m}_{\rho,\delta}=\Psi^{m}_{\rho,\delta}(G\times \widehat{G}):=\textnormal{Op}({S}^{m}_{\rho,\delta}(G\times \widehat{G})),$ for every $m\in \mathbb{R}.$ Then,
\begin{itemize}
    \item [(i)] The mapping $A\mapsto A^{*}:\Psi^{m}_{\rho,\delta}\rightarrow \Psi^{m}_{\rho,\delta}$ is a continuous linear mapping between Fr\'echet spaces and  the  symbol of $A^*,$ $\sigma_{A^*}(x,\pi)\equiv \widehat{A^{*}}(x,\pi) $ satisfies the asymptotic expansion,
 \begin{equation*}
    \widehat{A^{*}}(x,\pi)\sim \sum_{|\alpha|= 0}^\infty\Delta^{\alpha}X_x^\alpha (\widehat{A}(x,\pi)^{*}).
 \end{equation*} This means that, for every $N\in \mathbb{N},$ and all $\ell\in \mathbb{N},$
\begin{equation*}
   \Small{ \Delta^{\alpha_\ell}X_x^\beta\left(\widehat{A^{*}}(x,\pi)-\sum_{|\alpha|\leqslant N}\Delta^\alpha X_x^\alpha (\widehat{A}(x,\pi)^{*}) \right)\in {S}^{m-(\rho-\delta)(N+1)-\rho\ell+\delta[\beta]}_{\rho,\delta}(G\times\widehat{G}) },
\end{equation*} where $[\alpha_\ell]=\ell.$
\item [(ii)] The mapping $(A_1,A_2)\mapsto A_1\circ A_2: \Psi^{m_1}_{\rho,\delta}\times \Psi^{m_2}_{\rho,\delta}\rightarrow \Psi^{m_3}_{\rho,\delta}$ is a continuous bilinear mapping between Fr\'echet spaces, and the symbol of $A=A_{1}\circ A_2,$ satisfies the asymptotic expansion,
\begin{equation*}
    \sigma_A(x,\pi)\sim \sum_{|\alpha|= 0}^\infty(\Delta^\alpha\widehat{A}_{1}(x,\pi))(X_x^\alpha \widehat{A}_2(x,\pi)),
\end{equation*}this means that, for every $N\in \mathbb{N},$ and all $\ell \in\mathbb{N},$
\begin{align*}
    &\Delta^{\alpha_\ell}X_x^\beta\left(\sigma_A(x,\pi)-\sum_{|\alpha|\leqslant N}  (\Delta^{\alpha}\widehat{A}_{1}(x,\pi))(X_x^\alpha \widehat{A}_2(x,\pi))  \right)\\
    &\hspace{2cm}\in {S}^{m_1+m_2-(\rho-\delta)(N+1)-\rho\ell+\delta[\beta]}_{\rho,\delta}(G\times \widehat{G}),
\end{align*}for every  $\alpha_\ell \in \mathbb{N}_0^n$ with $[\alpha_\ell]=\ell.$
\item [(iii)] For  $0\leqslant \delta\leq  \rho\leqslant    1,$  $\delta\neq 1,$ let us consider a continuous linear operator $A:C^\infty(G)\rightarrow\mathscr{D}'(G)$ with symbol  $\sigma\in {S}^{0}_{\rho,\delta}(G\times \widehat{G})$. Then $A$ extends to a bounded operator from $L^2(G)$ to  $L^2(G).$ 
\end{itemize}
\end{theorem}
As for the $L^p$-boundedness theorem on graded Lie groups, we have the following statement (see \cite{CDR21JGA}).
\begin{theorem}\label{Lp1CardonaDelgadoRuzhansky2}
Let $G$ be a graded Lie group of homogeneous dimension $Q.$ Let $A:C^\infty(G)\rightarrow\mathscr{D}'(G)$ be a pseudo-differential operator with symbol $\sigma\in S^{-m}_{\rho,\delta}(G\times \widehat{G} ),$ $0\leq \delta\leq \rho\leq 1,$ $\delta\neq 1.$ If  $1<p<\infty,$ then $A=\textnormal{Op}(\sigma)$ extends to a bounded operator from $L^p(G)$ to  $L^p(G)$ provided that \begin{equation*}
   m\geq m_{p}:=Q(1-\rho)\left|\frac{1}{p}-\frac{1}{2}\right|.
\end{equation*}
\end{theorem}

\section{$L^p$-$L^q$-boundedness  on graded Lie groups}

\subsection{$L^p$-$L^q$-boundedness of Riesz and Bessel potentials} Here we discuss the $L^p$-$L^q$-boundedness of Riesz and Bessel potentials. We start with the case of Riesz potentials which already have been established in \cite{FischerRuzhanskyBook}. 

\begin{lemma}\label{Riesz} Let $\mathcal{R}$ be a positive Rockland operator of homogeneous degree $\nu>0$ on a graded Lie group $G.$ Then, the Riesz operator $I_{a}=\mathcal{R}^{-\frac{a}{\nu}},$ with $0<a<Q,$ admits a bounded extension from $L^p(G)$ into $L^q(G),$ that is, the estimate
\begin{equation}
    \Vert I_{a} f\Vert_{L^q}\leq  C\Vert  f\Vert_{L^p}\,\,\,
\end{equation}holds, if and only if, $p<q$ and
\begin{equation}\label{Necessary:condition}
    a=Q\left(\frac{1}{p}- \frac{1}{q}\right).
\end{equation}    
\end{lemma}
\begin{proof} 
    First, assume that $I_{a}:L^p(G)\rightarrow L^q(G)$ admits a bounded extension. Since $I_{a}$ is homogeneous of order $-a,$ Proposition 3.2.8 in \cite[Page 138]{FischerRuzhanskyBook} implies that the pair $(p,q)$ satisfies the conditions $p,q\in (1,\infty)$ together with \eqref{Necessary:condition}. That \eqref{Necessary:condition} and the condition $1<p<q<\infty$ are sufficient conditions for the existence of a  bounded extension  of $I_{a}:L^p(G)\rightarrow L^q(G)$ was proved in  \cite[Page 229]{FischerRuzhanskyBook}.
\end{proof}

\begin{lemma}\label {Lemma:Bessel:potential:lplq}Let $1<p,q<\infty.$ Let $\mathcal{R}$ be a positive Rockland operator of homogeneous degree $\nu>0$ on a graded Lie group $G.$ Then, the Bessel operator $B_{a}=(1+\mathcal{R})^{-\frac{a}{\nu}},$ admits a bounded extension from $L^p(G)$ into $L^q(G),$ that is, the estimate
\begin{equation}\label{eqref:lplq}
    \Vert B_{a} f\Vert_{L^q}\leq  C\Vert  f\Vert_{L^p}\,\,\,
\end{equation} holds, if and only if, $1<p<q<\infty$ and
\begin{equation}\label{Necessary:condition:2}
   a\geq Q\left(\frac{1}{p}- \frac{1}{q}\right).
\end{equation}    
\end{lemma}
\begin{proof} Before beginning the proof, let us first recall the following properties that will be used later in this proof:
\begin{itemize}
    \item \begin{equation} \label{p1}
        \mathscr{F}_G(f \circ D_r)(\pi)= r^{-Q} \mathscr{F}_G(f)(r^{-1} \cdot \pi)
    \end{equation} for any suitable functions $f,$ $r>0$ and $\pi \in \widehat{G}.$
    \item Let $\mathcal{R}$ be a positive Rockland operator of homogeneous degree $\nu>0$ then we have \begin{equation} \label{p2}
        ((r \cdot \pi)(\mathcal{R}))^{\frac{1}{\nu}}= r \pi(\mathcal{R})^{\frac{1}{\nu}}\,\,\,\text{for}\,\,r>0, \,\pi \in \widehat{G}.
    \end{equation}
\end{itemize}
The first property is easy to verify using the definition of the group Fourier transform.  For the proof of the second property, see \cite[Lemma 4.3]{FR21}. 

Now we begin the proof of Lemma \ref{Lemma:Bessel:potential:lplq}. First, let us prove that under the continuity property in \eqref{eqref:lplq} we have that $p<q$ and that \eqref{Necessary:condition:2} holds.   Assume that \eqref{eqref:lplq} holds with $1<p,q<\infty.$  That the inequality $p<q$ holds is a consequence of the general statement of Proposition 3.2.4 in \cite[Page 134]{FischerRuzhanskyBook}.
 Now, we are going to prove \eqref{Necessary:condition:2}. Let us consider the space of functions
\begin{equation}
  C_{0,\lambda}^\infty(G):=  E(\lambda,\infty)C^\infty_0(G),
\end{equation}where $\{dE_{\lambda}\}_{\lambda>0}$ denotes the spectral measure associated with the spectral projections $E(0,\lambda),$ $\lambda>0,$ of $\mathcal{R}$ and where  $E_{\lambda}:=E(\lambda,\infty)=I-E(0,\lambda).$  Note that $C_{0,\lambda}^\infty(G)$ is invariant under the action of the projection $E_{\lambda},$ in the sense that 
\begin{equation}\label{invariance}
 E(\lambda,\infty) C_{0,\lambda}^\infty(G)=C_{0,\lambda}^\infty(G).
\end{equation}Take $f\in C_{0,\lambda}^\infty(G) $ such that $f\neq 0.$ 
In view of \eqref{invariance} we have that $f=E(\lambda,\infty)f,$ and for some $\dot{f}\in C_{0}^\infty(G),$ $f=E(\lambda,\infty)\dot{f}.$ Note that we have the point-wise identity
\begin{equation}\label{Invariant:under:scale}
    \forall x\in G, f(x)=(E(\lambda,\infty)f)(x).
\end{equation}
By the duality of $(L^{q},L^{q'})$ we have that
\begin{align*}
    \Vert (1+\mathcal{R})^{-\frac{a}{\nu}} f \Vert_{L^q}=\sup_{\Vert g\Vert_{L^{q'}}=1 }|((1+\mathcal{R})^{-\frac{a}{\nu}} f,g)_{L^2} |.
\end{align*}Since $g=f/\Vert f\Vert_{L^{q'}}$ has norm less than or equal to one in $L^{q'}$ we have that
\begin{align*}
    \Vert (1+\mathcal{R})^{-\frac{a}{\nu}} f \Vert_{L^q}\geq |((1+\mathcal{R})^{-\frac{a}{\nu}} f,g)_{L^2} |=\frac{1}{\Vert f\Vert_{L^{q'}}}|((1+\mathcal{R})^{-\frac{a}{\nu}} f,f)_{L^2} |.
\end{align*}
Indeed, if $\Vert f\Vert_{L^{q'}}<\infty$ we have that $g=f/\Vert f\Vert_{L^{q'}}$ is a unit vector in $L^{q'}(G).$ In the case where $\Vert f\Vert_{L^{q'}}=\infty,$ then $g$ is the trivial function. In both cases, the positivity of $B_{a}=(1+\mathcal{R})^{-\frac{a}{\nu}}$ implies that
   $$  |((1+\mathcal{R})^{-\frac{a}{\nu}} f,f)_{L^2} |=((1+\mathcal{R})^{-\frac{a}{\nu}} f,f)_{L^2} \geq 0. $$
In what follows we assume that  $\lambda=2^{N}>0$ where $N$ is an integer. We are going to prove now that $\Vert f\Vert_{L^{q'}}<\infty.$ This can be deduced from the Littlewood-Paley theorem in  \cite{Nikolskii}. Indeed, if $\{\phi_j\}_{j\in \mathbb{Z}}$ is a dyadic partition of unity, in such a way that $\sum_{j}\phi_{j}(t)=1,$ with $\phi_{j}(t)=\phi_0(2^{-j}t),$ and $\phi_0\in \mathscr{S}(\mathbb{R})$ being positive, one has that for $1<q<\infty$ (and then $1<q'<\infty$),
\begin{align*}
    \Vert  f \Vert_{L^{q'}(G)} &= \Vert E(\lambda,\infty) f \Vert_{L^{q'}(G)}=\left\Vert\left(\sum_{j\in \mathbb{Z}}|\phi_{j}(\mathcal{R}^{1/\nu})E(\lambda,\infty)f|^2\right)^{\frac{1}{2}}\right\Vert_{ L^{q'}(G)}\\
    &=\left\Vert\left(\sum_{j\geq N}|\phi_{j}(\mathcal{R}^{1/\nu})E(\lambda,\infty)\dot{f}|^2\right)^{\frac{1}{2}}\right\Vert_{ L^{q'}(G)}\\
    &=\left\Vert\left(\sum_{j\geq N}|\phi_{j}(\mathcal{R}^{1/\nu})\dot{f}|^2\right)^{\frac{1}{2}}\right\Vert_{ L^{q'}(G)}\\
    &\leq \left\Vert\left(\sum_{j\in \mathbb{Z}}|\phi_{j}(\mathcal{R}^{1/\nu})\dot{f}|^2\right)^{\frac{1}{2}}\right\Vert_{ L^{q'}(G)}\\
    &\leq C\Vert \dot{f}\Vert_{L^{q'}}<\infty, \textnormal { since } \dot{f}\in C^{\infty}_0(G).
\end{align*}
By the commutativity of the functional calculus of $\mathcal{R}$ with the spectral projections $E(\lambda,\infty)$ we have that
\begin{align*}
  \Vert (1+\mathcal{R})^{-\frac{a}{\nu}} f \Vert_{L^q}&\geq \frac{1}{\Vert f\Vert_{L^{q'}}}  ((1+\mathcal{R})^{-\frac{a}{\nu}} f,f)_{L^2} \\&=\frac{1}{\Vert f\Vert_{L^{q'}}}\smallint_{G} (1+\mathcal{R})^{-\frac{a}{\nu}} f(x)\overline{f(x)}dx\\
   &=\frac{1}{\Vert f\Vert_{L^{q'}}}\smallint_{G}\smallint_{0}^{\infty} (1+\omega)^{-\frac{a}{\nu}}dE_\omega E(\lambda,\infty) f(x)\overline{f(x)}dx\\
   &=\frac{1}{\Vert f\Vert_{L^{q'}}}\smallint_{G}\smallint_{\lambda}^{\infty} (1+\omega)^{-\frac{a}{\nu}}dE_\omega  f(x)\overline{f(x)}dx\\
   &=\frac{1}{\Vert f\Vert_{L^{q'}}}\smallint_{\lambda}^{\infty} (1+\omega)^{-\frac{a}{\nu}} d \left(\smallint_{G}E_\omega  f(x)\overline{f(x)}dx\right)\\
   &=\frac{1}{\Vert f\Vert_{L^{q'}}}\smallint_{\lambda}^{\infty} (1+\omega)^{-\frac{a}{\nu}} d (E_\omega  f,\overline{f})_{L^2}.
\end{align*} In our further analysis note that there is a constant $C_{\lambda}>0$ such that 
\begin{equation}
    \forall \omega\geq \lambda ,\, (1+\omega)^{-\frac{a}{\nu}} \geq C_{\lambda} \omega^{-\frac{a}{\nu}}.
\end{equation}
Since $C_{\lambda}$ satisfies the inequality
\begin{align*}
    C_{\lambda}\leq (1+\omega)^{-\frac{a}{\nu}} \omega^{\frac{a}{\nu}}=\left(1-\frac{1}{1+\omega}\right)^{\frac{a}{\nu}}:=s(\omega),\,   \forall \omega\geq \lambda, 
\end{align*} and the function $s(\omega)$ is increasing when $\omega\geq \lambda,$ we can take
\begin{equation}\label{C:lambda}
    C_\lambda:=s(\lambda)=\left(1-\frac{1}{1+\lambda}\right)^{\frac{a}{\nu}}.
\end{equation}
In consequence, for all $f\in C_{0,\lambda}^\infty(G) $ such that $f\neq 0,$ we have the inequality
\begin{align*}
   \Vert (1+\mathcal{R})^{-\frac{a}{\nu}} f \Vert_{L^q} &\geq   \frac{1}{\Vert f\Vert_{L^{q'}}}\smallint_{\lambda}^{\infty} (1+\omega)^{-\frac{a}{\nu}} d (E_\omega  f,\overline{f})_{L^2}\\
             &\geq C_{\lambda}  \frac{1}{\Vert f\Vert_{L^{q'}}}\smallint_{\lambda}^{\infty} \omega^{-\frac{a}{\nu}} d (E_\omega  f,\overline{f})_{L^2}\\
             & =C_{\lambda}\frac{1}{\Vert f\Vert_{L^{q'}}}\smallint_{\lambda}^{\infty} \omega^{-\frac{a}{\nu}} d \left(\smallint_{G}E_\omega  f(x)\overline{f(x)}dx\right)\\
             &=C_{\lambda}\frac{1}{\Vert f\Vert_{L^{q'}}}\smallint_{G}\smallint_{\lambda}^{\infty} \omega^{-\frac{a}{\nu}}dE_\omega  f(x)\overline{f(x)}dx\\
             &= C_\lambda \frac{1}{\Vert f\Vert_{L^{q'}}}  (\mathcal{R}^{-\frac{a}{\nu}} f,f)_{L^2}.
\end{align*} Now, assume that $(1+\mathcal{R})^{-\frac{a}{\nu}}$ admits a bounded extension from $L^p(G)$ to $L^q(G).$ Then, there exists $C>0,$ such that, in particular, for all $f\in C_{0,\lambda}^\infty(G) $ such that $f\neq 0,$ we have the inequality

\begin{equation}
    C\Vert f\Vert_{L^p(G)}\geq \Vert (1+\mathcal{R})^{-\frac{a}{\nu}} f \Vert_{L^q}.
\end{equation}Combining the previous inequality and the analysis above we have that
\begin{equation}\label{auxiliar}
    C\Vert f\Vert_{L^p(G)}\geq C_\lambda \frac{1}{\Vert f\Vert_{L^{q'}}}  (\mathcal{R}^{-\frac{a}{\nu}} f,f)_{L^2}.
\end{equation}
Now we will use a scaling argument. Let $r=2^{-M},$ where $M$ is an integer, and let us analyse the action of the projection $E(\lambda,\infty)$ on the function $f\circ D_{r}.$ Note that
\begin{align*}
    [E(\lambda,\infty)(f\circ D_r)](x) &= [E(\lambda,\infty)\left(\sum_{j\in \mathbb{Z}}\phi_{j}(\mathcal{R}^{1/\nu})\right) (f\circ D_r)](x)\\
    &= \left(\sum_{j\geq N}\phi_{j}(\mathcal{R}^{1/\nu}) (f\circ D_r)\right)(x).
\end{align*} We observe that, for any $z\in G,$ using \eqref{p1} and \eqref{p2}, we have
\begin{align*}
    \phi_{j}(\mathcal{R}^{1/\nu})(f\circ D_r)(z) &=\smallint\limits_{\widehat{G}}\textnormal{Tr}[\pi(z)\phi_{j}(\pi(\mathcal{R})^{1/\nu})\mathscr{F}_{G}[f\circ D_r](\pi)]d\pi\\
    &= \smallint\limits_{\widehat{G}}\textnormal{Tr}[\pi(z)\phi_{j}(\pi(\mathcal{R})^{1/\nu})\widehat{f}(r^{-1}\cdot \pi)]r^{-Q}d\pi\\
     &= \smallint\limits_{\widehat{G}}\textnormal{Tr}[((r\cdot \pi')(z))\phi_{j}( ((r\cdot\pi')(\mathcal{R}))^{1/\nu})\widehat{f}( \pi')]d\pi'\\
     &=\smallint\limits_{\widehat{G}}\textnormal{Tr}[\pi'(D_{r}(z)))\phi_{j}( ((r\cdot\pi')(\mathcal{R}))^{1/\nu})\widehat{f}( \pi')]d\pi'\\
     &= \smallint\limits_{\widehat{G}}\textnormal{Tr}[\pi'(D_{r}(z)))\phi_{0}( 2^{-j}((r\cdot\pi')(\mathcal{R}))^{1/\nu})\widehat{f}( \pi')]d\pi'\\
     &= \smallint\limits_{\widehat{G}}\textnormal{Tr}[\pi'(D_{r}(z)))\phi_{0}( 2^{-j}((r^\nu\pi'(\mathcal{R}))^{1/\nu})\widehat{f}( \pi')]d\pi'\\
     &= \smallint\limits_{\widehat{G}}\textnormal{Tr}[\pi'(D_{r}(z)))\phi_{0}( 2^{-j}r\pi'(\mathcal{R})^{1/\nu})\widehat{f}( \pi')]d\pi'\\
     &= \smallint\limits_{\widehat{G}}\textnormal{Tr}[\pi'(D_{r}(z)))\phi_{0}( 2^{-j-M}\pi'(\mathcal{R})^{1/\nu})\widehat{f}( \pi')]d\pi'\\
     &=(\psi_{j+M}(\mathcal{R}^{\frac{1}{\nu}})f)\circ D_r(z).
\end{align*}
In consequence
\begin{align*}
     [E(\lambda,\infty)(f\circ D_r)](x) &= \left(\sum_{j\geq N}\phi_{j}(\mathcal{R}^{1/\nu}) (f\circ D_r)\right)(x)\\
     &= \left(\sum_{j\geq N}(\phi_{j+M}(\mathcal{R}^{\frac{1}{\nu}})f)\circ D_r\right)(x)\\
      &= \left(\sum_{j\geq N+M}(\phi_{j}(\mathcal{R}^{\frac{1}{\nu}})f)\circ D_r\right)(x)
      \\
     &=(f-E(0,2^{N+M})f)\circ D_r(x)\\
     &=(E(2^{N+M},\infty)f)\circ D_r.
\end{align*}
We have proved that 
\begin{align}\label{eq:dila}
  \forall f\in C^{\infty}_{0,\lambda}(G), \, \forall r=2^{-M}, \,  E(\lambda,\infty)(f\circ D_r)= (E(2^{N+M},\infty)f)\circ D_{r}. 
\end{align}
Let us consider the set of dilations $$\mathbb{D}_0:=\{D_r:r=2^{-M},\,M\in \mathbb{Z}\}.$$ 
In what follows let us consider $R\in \mathbb{D}_0,$ and let us consider the function $f\circ D_R\in C_{0,\lambda}^\infty(G).$ By applying \eqref{auxiliar} with $E(2^{N},\infty)(f\circ D_R),$  we have the inequality
\begin{align*}
     &C\Vert E(2^{N},\infty)(f\circ D_R)\Vert_{L^p(G)}\\
     &\geq C_\lambda \frac{1}{\Vert E(2^{N},\infty)(f\circ D_R)\Vert_{L^{q'}}}  (\mathcal{R}^{-\frac{a}{\nu}} E(2^{N},\infty)(f\circ D_R),E(2^{N},\infty)(f\circ D_R))_{L^2}.
\end{align*}Using that
\begin{itemize}
    \item $\Vert (E(2^{N+M},\infty)f)\circ D_R \Vert_{L^p(G)}=R^{-Q/p}\Vert E(2^{N+M},\infty)f\Vert_{L^p(G)},$
    \item $\Vert (E(2^{N+M},\infty)f)\Vert_{L^{q'}(G)}=R^{-Q/q'}\Vert E(2^{N+M},\infty)f\Vert_{L^{q'}(G)},$
\end{itemize}  that the change of variables  $y=D_{R}(x)$ gives the volume element $dx=R^{-Q}dy,$ and the identities below
\begin{align*}
    &(\mathcal{R}^{-\frac{a}{\nu}} (E(2^{N+M},\infty)f)\circ D_R ,(E(2^{N+M},\infty)f)\circ D_R )_{L^2}\\
    &=\smallint_G \mathcal{R}^{-\frac{a}{\nu}}(E(2^{N+M},\infty)f)\circ D_R (x)\overline{(E(2^{N+M},\infty)f)\circ D_R (x)}dx\\
    &=\smallint_G R^{-a}(\mathcal{R}^{-\frac{a}{\nu}}E(2^{N+M},\infty)f)\circ D_R(x))\overline{E(2^{N+M},\infty)f\circ D_R(x)}dx\\
    &= R^{-a-Q}\smallint_G (\mathcal{R}^{-\frac{a}{\nu}}E(2^{N+M},\infty)f)(y))\overline{E(2^{N+M},\infty)f(y)}dy\\
    &= R^{-a-Q}(\mathcal{R}^{-\frac{a}{\nu}}E(2^{N+M},\infty)f,E(2^{N+M},\infty)f)_{L^2},
\end{align*}we have that for any $R\in \mathbb{D}_0,$ and with $h=E(2^{N+M},\infty)f,$
we get the following lower-bound estimate
\begin{align*}
  C  R^{-Q/p}\Vert h\Vert_{L^p(G)}\geq C_{\lambda}R^{Q/q'}\frac{1}{\Vert h\Vert_{L^{q'}}}  R^{-a-Q}(\mathcal{R}^{-\frac{a}{\nu}}h,h)_{L^2}.
\end{align*}Let us observe that the previous inequality is equivalent to the following one
\begin{equation*}
   C R^{-Q/p+a+Q-Q/q'}\Vert h\Vert_{L^p(G)}\Vert h\Vert_{L^{q'}}C_{\lambda}^{-1} \geq  (\mathcal{R}^{-\frac{a}{\nu}}h,h)_{L^2}>0.
\end{equation*}Using \eqref{C:lambda} we have the inequality
\begin{equation}\label{eq:auxiliar}
   C R^{-Q/p+a+Q-Q/q'}\Vert h\Vert_{L^p(G)}\Vert h\Vert_{L^{q'}}\left(1-\frac{1}{1+2^{N}}\right)^{-\frac{a}{\nu}} \geq  (\mathcal{R}^{-\frac{a}{\nu}}h,h)_{L^2}>0, \,R=2^{-M}.
\end{equation}
   If $\varkappa:=-Q/p+a+Q-Q/q'<0,$ we take $M<0,$ $N=-M.$ Then $ h:=E(2^{N+M},\infty)f=E(1,\infty)f$ is independent of the pair $(N,M)=(-M,M)$ and consequently  \eqref{eq:auxiliar} takes the form
    \begin{equation}\label{case:1}
        C 2^{-M\varkappa}\Vert h\Vert_{L^p(G)}\Vert h\Vert_{L^{q'}}\left(1-\frac{1}{1+2^{-M}}\right)^{-\frac{a}{\nu}} \geq  (\mathcal{R}^{-\frac{a}{\nu}}h,h)_{L^2}>0. 
    \end{equation}
    Since $M<0,$ $2^{-M}>1,$ we have
\begin{align*}
    \left(1-\frac{1}{1+2^{-M}}\right)^{-\frac{a}{\nu}}=2^{\frac{Ma}{\nu}}(1+2^{-M})^{\frac{a}{\nu}}\leq 2^{\frac{Ma}{\nu}}(2\times 2^{-M})^{\frac{a}{\nu}} =2^{\frac{a}{\nu}}.
\end{align*}In consequence, letting $-M\rightarrow \infty,$ $2^{-M\varkappa}\rightarrow 0^{+}$ implying that the inequality in \eqref{case:1} is only possible if 
\begin{equation}
    \varkappa=-Q/p+a+Q-Q/q'=a-Q/p+Q/q\geq 0,
\end{equation} 
 as desired.
Now, in order to prove the converse, that is if \eqref{Necessary:condition:2} holds, then $B_a$ is bounded from $L^p(G)$ into $L^{q}(G),$ let us proceed as follows. If 
$$ a=Q(1/p-Q/q), $$
the boundedness of $B_a$  from $L^p(G)$ into $L^{q}(G)$ is a consequence of the Sobolev embedding theorem  proved in \cite[page 246]{FischerRuzhanskyBook}. Indeed, from \cite[Theorem 4.4.28]{FischerRuzhanskyBook} it follows that the continuous embedding $L^{p}_v(G)\subset L^{p}_u(G) $ holds for $v-u=Q(1/p-Q/q).$ Certainly, this means that
\begin{equation}
    \forall f\in \mathscr{S}(G),\,\,\Vert B_{-u} f\Vert_{L^q(G)}\leq  C \Vert B_{-v} f\Vert_{L^p(G)}.
\end{equation}With $a=-u$ and $v=0,$ that is when $a=Q(1/p-Q/q),$ we have the boundedness operator inequality
 \begin{equation}
    \forall f\in \mathscr{S}(G),\,\,\Vert B_{a} f\Vert_{L^q(G)}\leq C \Vert f\Vert_{L^p(G)}.
\end{equation}Now, if $a>Q(1/p-Q/q),$ we take $\varepsilon=a-Q(1/p-Q/q)>0,$ and we factorise
\begin{equation}
    B_a=(1+\mathcal{R})^{-\frac{a}{\nu}}=(1+\mathcal{R})^{-Q(1/p-1/q)/\nu-\varepsilon/\nu}= (1+\mathcal{R})^{-Q(1/p-Q/q)}B_{\varepsilon}.
\end{equation}Is clear that $(1+\mathcal{R})^{-Q(1/p-Q/q)/\nu}:L^{p}(G)\rightarrow L^{q}(G)$ is bounded. On the other hand $B_{\varepsilon}:L^{p}(G)\rightarrow L^{p}(G)$ is bounded on $L^p(G)$ for all $1<p<\infty.$ Indeed, $$B_{\varepsilon}\in \Psi^{-\varepsilon}_{1,0}(G\times \widehat{G})\subset \mathscr{B}(L^p(G)).$$ In consequence $ B_a$ is bounded from $L^p(G)$ into $L^q(G).$ The proof is complete.
\end{proof}
\begin{remark} Observe that the necessary and sufficient conditions of  Lemma
    \ref{Riesz} have been obtained in \cite{FischerRuzhanskyBook}. While in Lemma \ref{Lemma:Bessel:potential:lplq}  that $1<p<q<\infty$ is a necessary condition for the $L^p$-$L^q$-boundedness of $B_a,$ this fact follows from a more general statement, see Proposition 3.2.4 in \cite[Page 134]{FischerRuzhanskyBook}:

    {\it 
    if $T:L^p(G)\rightarrow L^q(G)$ is  left-invariant and bounded, with $1\leq p,q<\infty,$ and $p> q,$ then $T=0.$ }
    
On the other hand that  
\eqref{Necessary:condition:2} is a necessary condition for the $L^p$-$L^q$-boundedness of $B_a$ is the contribution to the statement in Lemma  \ref{Lemma:Bessel:potential:lplq} in this paper.
\end{remark}

\subsection{$L^p$-$L^q$-boundedness of pseudo-differential operators I}
The aim of this subsection is to extend the results in the previous subsection to the pseudo-differential setting. The following result presents the necessary and sufficient criteria for a pseudo-differential operator to be bounded from $L^p(G)$ into $L^q(G)$ for the range $1<p\leq 2 \leq q<\infty.$

\begin{theorem} \label{thh1}
    
Let $1<p\leq 2\leq q<\infty$ and $m \in \mathbb{R}.$ Let $G$ be a graded Lie group of  homogeneous dimension $Q.$ Then, every pseudo-differential operator $A\in \Psi^{m}_{\rho,\delta}(G\times \widehat{G})$ with $0\leq \delta \leq  \rho \leq 1$ and $\delta \neq 1,$ admits a bounded extension from $L^p(G)$ into $L^q(G),$ that is
\begin{equation}
   \forall f\in C^{\infty}_0(G),\, \Vert A f\Vert_{L^q}\leq  C\Vert  f\Vert_{L^p}\,\,\,
\end{equation} holds, if and only if, 
\begin{equation}\label{Necessary:condition:3a}
   m\leq -Q\left(\frac{1}{p}- \frac{1}{q}\right).
\end{equation}    
\end{theorem}
\begin{proof}
    Assume that $m> -Q\left(\frac{1}{p}- \frac{1}{q}\right).$ We are going to show that there exists $A\in \Psi^{m}_{\rho,\delta}(G\times \widehat{G})$ which is not bounded from $L^p(G)$ into $L^q(G).$ We consider $$A=B_{-m}=(1+\mathcal{R})^{\frac{m}{\nu}}\in \Psi^{m}_{1,0}(G\times \widehat{G})\subset \Psi^{m}_{\rho,\delta}(G\times \widehat{G}).$$ Since $$-m< Q\left(\frac{1}{p}- \frac{1}{q}\right),$$ from Lemma \ref{Lemma:Bessel:potential:lplq}, we have that  $A=B_{-m}$ is not bounded from $L^p(G)$ into $L^q(G).$ So, we have proved the necessity of the order condition  \eqref{Necessary:condition:3a}. Now, in order to prove the reverse statement, we consider $m$ satisfying \eqref{Necessary:condition:3a} and  $m_1$ and $m_2$ satisfying the conditions
    \begin{equation}
     m=m_1+m_2,\,   m_1\leq -Q(1/p-1/2),\,m_2\leq -Q(1/2-1/q).
    \end{equation} If  $A\in \Psi^{m}_{\rho,\delta}(G\times \widehat{G}),$ we factorise $A$ as follows,
    \begin{align*}
        A=B_{-m_2}A_0 B_{-m_1},\,\,A_0=B_{m_2}AB_{m_1}.
    \end{align*}Note that $A_0\in \Psi^{0}_{\rho,\delta}(G\times \widehat{G}).$ The Calder\'on-Vaillancourt theorem (Theorem \ref{calculus}(iii)) implies that $A_0$ is bounded from $L^2(G)$ into $L^2(G).$ On the other hand, from Lemma \ref{Lemma:Bessel:potential:lplq} we have that $B_{m_2}:L^2(G)\rightarrow L^q(G),$ and $B_{m_1}:L^p(G)\rightarrow L^2(G),$ are bounded operators. In consequence, we have proved that $A$ admits a bounded extension from $L^p(G)$ into $L^q(G).$ The proof is complete.    
\end{proof}

\subsection{$L^p$-$L^q$-boundedness of pseudo-differential operators II}
In this subsection, we consider the $L^p-L^q$ boundedness of pseudo-differential operators on graded Lie groups for a wider range of indices $p$ and $q.$ Our main result of this section is the following theorem. 
\begin{theorem} \label{thh2} Let $1<p \leq q <\infty, m \in \mathbb{R},$ and let  $G$ be a graded Lie group of  homogeneous dimension $Q.$ Then, every pseudo-differential operator $A\in \Psi^{m}_{\rho,\delta}(G\times \widehat{G})$ with $0\leq \delta \leq  \rho \leq 1$ and $\delta \neq 1,$ admits a bounded extension from $L^p(G)$ into $L^q(G),$ that is
\begin{equation}
   \forall f\in C^{\infty}_0(G),\, \Vert A f\Vert_{L^q}\leq  C\Vert  f\Vert_{L^p}\,\,\,
\end{equation} holds in the following cases: 
\begin{itemize}
        \item[(i)] if $1<p\leq q \leq 2$ and  \begin{equation}\label{Necessary:condition:4}
   m\leq -Q \left( \frac{1}{p}-\frac{1}{q}+(1-\rho) \left(\frac{1}{q}-\frac{1}{2}\right)\right). \end{equation}    
\item[(ii)] if $2 \leq p \leq q<\infty$ and 

\begin{equation}\label{Necessary:condition:3}
   m\leq -Q\left( \frac{1}{p}-\frac{1}{q}+(1-\rho) \left(\frac{1}{2}-\frac{1}{p}\right)\right);
   \end{equation} 
\end{itemize}
\end{theorem}
\begin{proof}  
\begin{itemize}
    \item[(i)] Let us consider $p, q$ and $m$ satisfying the conditions given in (i). Choose $m'= -Q \left( \frac{1}{p}-\frac{1}{q} \right)$ and this implies that the Bessel potential $B_{-m'}$ is bounded from $L^p(G)$ to $L^q(G)$ as a consequence of Lemma \ref{Lemma:Bessel:potential:lplq}. For $A \in \Psi^m_{\rho, \delta}(G \times \widehat{G}),$ we decompose it as follows:
    $$A= (A B_{m'}) B_{-m'}.$$
    Now, we note that operator $AB_{m'} \in \Psi^{m-m'}_{\rho, \delta}(G \times \widehat{G})$ with $m-m'$ satisfying $m-m' \leq -Q(1-\rho) \left( \frac{1}{q}-\frac{1}{2} \right).$ Then, Theorem \ref{Lp1CardonaDelgadoRuzhansky2} shows that $AB_{m'}$ is bounded operator from $L^q(G)$ into $L^q(G).$ Therefore, we conclude that the operator $A$ has a bounded extension from $L^p(G)$ into $L^q(G).$
    \item[(ii)] To prove this part we follow the same strategy as in Part (i). We factorise the operator $A \in \Psi^m_{\rho, \delta}(G \times \widehat{G})$  in the following manner:
    $$A= B_{-m'} (B_{m'} A),$$
    where $m'=-Q(\frac{1}{p}-\frac{1}{q}).$ Again, it follows from Lemma \ref{Lemma:Bessel:potential:lplq} that the operator $B_{-m'}$ is a bounded from $L^{p}(G)$ into $L^q(G).$ On the other hand, the operator $B_{m'} A \in \Psi^{m-m'}_{\rho, \delta}(G \times \widehat{G})$ with $m-m'\leq -Q(1-\rho) \left(\frac{1}{2}-\frac{1}{p} \right),$ which, as a consequence of Theorem \ref{Lp1CardonaDelgadoRuzhansky2}, yields that the operator $B_{-m'} A$ is bounded from $L^p(G)$ into $L^p(G).$ Hence, we conclude that the operator $A$ has a bounded extension from $L^p(G)$ into $L^q(G).$ 
\end{itemize}This completes the proof of this theorem.  \end{proof}
Now, we illustrate with an example our main Theorem \ref{main:theorem}.
We end this section with the following example. 
\begin{example} Let $L_1,L_2\in C^\infty(G)$ be complex-valued smooth functions with bounded derivatives, i.e. such that for any $\alpha,$ $X_{x}^\alpha L_1$ and $X_{x}^\alpha L_2$ are bounded functions and such that for some $\Lambda\geq 0,$ there exists a constant $c_\Lambda>0$ such that they satisfy the following {\it growth condition }
\begin{equation}\label{G;condition}
  \forall x\in G, \,\forall \lambda\geq \Lambda,\,  |L_1(x)+L_2(x)\lambda|\geq c_\Lambda(1+\lambda).
\end{equation}Then the operator
$$ A(x,\mathcal{R}):=L_1(x)+L_2(x)\mathcal{R}\in \Psi^{\nu}_{1,0}(G\times \widehat{G})$$ is a non-invariant elliptic operator with a left-parametrix $P(x,\mathcal{R})\in \Psi^{-\nu}_{1,0}(G\times \widehat{G})$ (see \cite[Proposition 5.8.2]{FischerRuzhanskyBook} and \cite{CDRFC}). According to the order condition in \eqref{order:condition}, 
if 
\begin{equation}\label{eq:ref:example}
    \nu\geq Q\left(   \frac{1}{p}-\frac{1}{q}\right),
\end{equation}where $Q$ is the homogeneous dimension of the group $G,$   $$P(x,\mathcal{R}):L^p(G)\rightarrow L^q(G)$$ is bounded. 

We note that the inequality in \eqref{eq:ref:example} is not sharp in the sense that the difference
    $$\nu- Q\left(   \frac{1}{p}-\frac{1}{q}\right)$$ could be strictly positive. However, if $r,a\in \mathbb{R}$ are real parameters in such a way that
    $$r-a=\nu- Q\left(   \frac{1}{p}-\frac{1}{q}\right),$$ we will use this fact in Theorem \ref{application} in the next section to establish the subelliptic regularity of elliptic operators.
\end{example}

\section{Applications}
\subsection{Subelliptic  regularity on graded Lie groups} This subsection is devoted to applying the $L^p$-$L^q$-boundedness of pseudo-differential operators on graded Lie groups to the validity of subelliptic estimates for elliptic operators. For the main aspects of this subject in the Euclidean setting and on compact manifolds we refer to H\"ormander \cite{Hormander1967,HormanderBook34} and Kohn and Nirenberg \cite{KN65}, see also Taylor \cite{Tay}.
\begin{theorem}\label{application}
    Let $L_1,L_2\in C^\infty(G)$ be complex functions satisfying the growth property in \eqref{G;condition}. Let $\mathcal{R}$ be a positive Rockland operator of homogeneous degree $\nu>0,$ and assume that $u\in \mathscr{S}'(G)$ and $f\in C^\infty(G)$ are solutions of the equation
    \begin{equation}
        A(x,\mathcal{R})u=f,\,  A(x,\mathcal{R}):=L_1(x)+L_2(x)\mathcal{R}.
    \end{equation}If $f\in L^p_r(G),$ for some $r\in \mathbb{R},$ with $1<p<\infty,$ then for any $s\in \mathbb{R},$ there exists $C_s>0,$ such that the {\it a-priori} estimate
    \begin{equation}
        \Vert u\Vert_{L^q_a}\leq C_s(\Vert f\Vert_{L^p_r(G)}+\Vert u \Vert_{L^q_{-s}(G)}),
    \end{equation}holds for $a \in \mathbb{R}$ provided that $1<p\leq q<\infty$ and that
    \begin{equation}
    \nu\geq a-r+ Q\left(   \frac{1}{p}-\frac{1}{q}\right).
\end{equation}
    \end{theorem}
    \begin{proof}Let $P:=P(x,\mathcal{R})$ be a left-parametrix of $A:=A(x,\mathcal{R}).$ Then, we have that
    $S:=PA-I\in \Psi^{-\infty}_{1,0}(G\times \widehat{G}). $ Let $a, r,  s\in \mathbb{R}.$ From the identity $Au=f,$ we have that
    \begin{equation}
        PA u= u+Su=Pf,
    \end{equation}and then
    \begin{equation}
        \Vert B_{-a} u\Vert_{L^q(G)}=\Vert B_{-a}Pf-B_{-a}Su\Vert_{L^q(G)}\leq \Vert B_{-a}PB_{r}B_{-r} u\Vert_{L^q(G)}+\Vert B_{-s}S u \Vert_{L^q(G)}. 
    \end{equation}Since the operator $B_{-a}S$ is smoothing, we have that, for any $s \in \mathbb{R},$ there exists $C_s>0$ such that the following inequality holds
    $$ \Vert B_{-a}S u \Vert_{L^q(G)}= \Vert B_{-a}SB_{-s}B_{s} u \Vert_{L^q(G)}\leq C_s\Vert B_{s}u\Vert_{L^q(G)}= C_s\Vert u\Vert_{L^q_{-s}(G)}.  $$ On the other hand, the operator $B_{-a}PB_{r}$ has order $a-\nu-r$ and it is bounded from $L^p(G)$ into $L^q(G)$ (see \eqref{order:condition} and Remark \ref{rem1}) if 
    $$ -a+\nu+r\geq Q\left(   \frac{1}{p}-\frac{1}{q}\right). $$ In this case we have that 
    \begin{align*}
       \Vert u\Vert_{L^q_a(G)}=  \Vert B_{-a} u\Vert_{L^q(G)}\lesssim_{s} \Vert B_{-r}u\Vert_{L^p(G)}+\Vert u\Vert_{L^q_{-s}(G)}=\Vert u\Vert_{L^p_r(G)}+\Vert u\Vert_{L^q_{-s}(G)}.
    \end{align*}The proof is complete.        
    \end{proof}
\subsection{$L^p$-$L^q$-bounds for $\tau$-quantisations on graded Lie groups}
This subsection is devoted to discussing the implications of our results in the setting of newly developed $\tau$-quantisation on graded Lie groups by the third author with S. Federico and D. Rottensteiner \cite{FRR23} (see also \cite{MR17} for general locally compact groups).  We refer to \cite{FRR23} for more details on the $\tau$-quantisation and several other results including the global symbolic calculus for this general quantisation. 

We will here briefly recall the relationship between the $\tau$-quatisation and the global Kohn-Nirenberg quantisation on graded Lie groups. 

We say that a measurable function $\tau: G \rightarrow G$ often referred as {\it quantising function} satisfies the property (HP) if, by representing it via exponential coordinates as $\tau(x)=(c_1^\tau(x), \ldots, c_n^\tau(x)) \in G,$ either the coordinates $c_i^\tau,\,i=1,\ldots, n,$ of $\tau(x)$ vanish identically or  are homogeneous polynomials given by     
$$c_i^\tau(x):=c_i^\tau(x_1,\ldots, x_i)=C_i^\tau x_i+d_i^\tau(x_1, \ldots, x_{i-1})$$
for some nonzero $C_i^\tau$ and some homogeneous polynomials $d_i^\tau$ of degree $\nu_i$ depending only on $x_k$ for $k=1, \ldots, i-1.$ Having this quantisation function $\tau,$ one defines the $\tau$-quantisation on a graded Lie group $G$ by
\begin{equation}
    \text{Op}^{\tau}(\sigma) f(x)= \int_{\widehat{G}} \text{Tr} \left( \int_G \pi(y^{-1}x) \sigma(x \tau (y^{-1}x)^{-1}, \pi) f(y)\,  dy \right) d\pi.  
\end{equation}
In the following  result of \cite{FRR23}, the authors established the relationship between the Kohn-Nirenberg quantisation and the $\tau$-quantisation on graded Lie groups.
\begin{theorem} \cite[Theorem 4.12]{FRR23} \label{dsr}
    Let $\tau$ be quantising function different from the constant function $\tau \equiv e_G$ on $G$ which satisfies the property (HP) and let $A$ be a continuous operator from $\mathscr{S}(G)$ to $\mathscr{S}(G)$. Assume that $m \in \mathbb{R}$ and $0 \leq \delta <\min\{\rho, \frac{1}{\nu_n}\}\leq 1$  and there exist two symbols $\sigma$ and $\sigma_\tau$  such that $A= \text{Op}^\tau(\sigma_\tau)= \text{Op}(\sigma).$ Then $\sigma \in S^m_{\rho, \delta}(G \times \widehat{G})$ if and only if $\sigma_\tau \in S^m_{\rho, \delta}(G \times \widehat{G}).$
\end{theorem}
By combining Theorem \ref{dsr}, Theorem \ref{thh1} and Theorem \ref{thh2} we obtain the following $L^p$-$L^q$ boundedness result for the $\tau$-quantisation of operators on $G.$
\begin{theorem} \label{w}
    Let $m \in \mathbb{R},$ $0 \leq \delta <\min\{\rho, \frac{1}{\nu_n}\} \leq 1,$ and let $\tau$ be a quantising function satisfying (HP). Then, for $1<p \leq q <\infty,$ the operator $T =\text{Op}^\tau(\sigma)$ has a bounded extension from $L^p(G)$ to $L^q(G)$ for $\sigma \in S^m_{\rho, \delta}(G \times \widehat{G})$ if 
    $$m\leq -Q\left(   \frac{1}{p}-\frac{1}{q}+(1-\rho)\max\left\{ \frac{1}{2}-\frac{1}{p},\frac{1}{q}-\frac{1}{2},0\right\}\right).$$ 
    Moreover, every operator $\text{Op}^\tau(\sigma)$ has a bounded extension from $L^p(G)$ to $L^q(G)$ for the range $1<p \leq 2 \leq q <\infty$ if and only if 
    $m \leq -Q\left( \frac{1}{p}-\frac{1}{q} \right).$
\end{theorem}
\begin{proof}
    It follows from Theorem \ref{dsr} that the operator $A=\text{Op}^\tau(\sigma)$ can be written as $A=\text{Op}(\sigma')$ for some $\sigma' \in S^m_{\rho, \delta}(G \times \widehat{G})$ in a unique way. Hence,  the assertion of theorem immediately   follows  from Theorem \ref{thh1} and Theorem \ref{thh2}. 
\end{proof}

\bibliographystyle{amsplain}

\end{document}